\documentclass[preprint,11pt]{elsarticle}
\usepackage{lineno,hyperref}
\modulolinenumbers[5]

\usepackage{color}
\usepackage{amsmath}
\usepackage{amsfonts,amsthm,amssymb}
\usepackage{amsfonts}
\usepackage{graphics}
\usepackage{pstricks}
\usepackage{graphicx}
\usepackage{cleveref}
\usepackage{extarrows,chngpage,array,float}
\usepackage{amssymb}
\usepackage{latexsym,bm}
\newtheorem{theorem}{Theorem}[section]

\newtheorem{corollary}[theorem]{Corollary}

\allowdisplaybreaks   

\numberwithin{equation}{section}
\journal{}

\begin{document}

\begin{frontmatter}

\title{Convolution formulas for multivariate arithmetic Tutte polynomials}

\author{Tianlong Ma,  \ \ Xian'an Jin\footnote{Corresponding author.},\ \ Weiling Yang\\
\small School of Mathematical Sciences\\[-0.8ex]
\small Xiamen University\\[-0.8ex]
\small P. R. China\\
\small\tt Email: tianlongma@aliyun.com, xajin@xmu.edu.cn, ywlxmu@163.com}

\begin{abstract}
The multivariate arithmetic Tutte polynomial of arithmetic matroids is a generalization of the multivariate Tutte polynomial of matroids. In this note, we give the convolution formulas for the multivariate arithmetic Tutte polynomial of the product of two arithmetic matroids. In particular,
the convolution formulas for the multivariate arithmetic Tutte polynomial
of an arithmetic matroid are obtained. Applying our results,
several known convolution formulas including \cite[Theorem 10.9 and Corollary 10.10]{Dupont} and \cite[Theorems 1 and 4]{Backman} are proved by a purely combinatorial
proof.  The proofs presented here are significantly shorter than the previous ones.
In addition, we obtain a convolution formula for the characteristic polynomial of an arithmetic matroid. 
\end{abstract}


\end{frontmatter}

\section{Introduction}
Let $M=(X,rk)$ denote a matroid on ground set $X$ with rank function $rk$. For each $e\in X$, let $v_e$ be a variable of $e$, and for $A\subseteq X$, let $\mathbf{v}_A=\{v_e:e\in A\}$. In particular, we write $\mathbf{v}$ for $\mathbf{v}_X$. In \cite{Sokal}, Sokal defined the
following multivariate version of the Tutte polynomial of a matroid $M=(X,rk)$ in the variables $q^{-1}$ and $\mathbf{v}$: 
\begin{align}\label{equ1.1}
	Z_{M}(q,\mathbf{v})=\sum_{A\subseteq X}q^{-rk(A)}\prod_{e\in A}v_e. 
\end{align}
In particular, if $v_e=v$ for each $e\in X$ in a matroid $M=(X,rk)$, then we write $Z_{M}(q,v)$ for $Z_{M}(q,\mathbf{v})$.

If we substitute $q$ and the variable $v_e$ of each $e\in X$ for $(x-1)(y-1)$ and $y-1$ respectively, and multiply by a prefactor $(x-1)^{r(M)}$ into Eq. (\ref{equ1.1}),
we obtain the standard bivariate Tutte polynomial:
\[T_{M}(x,y)=\sum_{A\subseteq X}(x-1)^{rk(X)-rk(A)}(y-1)^{|A|-rk(A)}.\]

A triple $\mathcal{M}=(X,rk,m)$ is called a \textit{multiplicity matroid}, introduced by Moci in \cite{Moci}, if $(X,rk)$ is a matroid and $m$ is a function (called multiplicity) from the family of all subsets of $X$ to the positive integers, that is, $m: 2^{X}\rightarrow \mathbb{Z}_{> 0}$.
We say that $m$ is \textit{trivial} if it is identically equal to $1$.

A multiplicity matroid $\mathcal{M}=(X,rk,m)$ is called an \textit{arithmetic matroid}, introduced by D'Adderio and Moci in \cite{D'AdderioMoci}, if $m$ satisfies the following axioms:
\begin{description}
	\item[(1)] For all $A\subseteq X$ and $e \in X$, if $rk(A\cup \{e\})=rk(A)$, then $m(A\cup \{e\})$ divides $m(A)$; otherwise, $m(A)$ divides $m(A\cup \{e\})$.
	\item[(2)] If $A\subseteq B \subseteq X$ and $B$ is a disjoint union $B=A\cup F\cup T$ such that for all $A\subseteq C\subseteq B$ we have $rk(C)=rk(A)+|C\cap  F|$, then
	$m(A)\cdot m(B)=m(A\cup F)\cdot m(A\cup T).$
	\item[(3)]  If $A\subseteq B \subseteq X$ and $rk(A)=rk(B)$, then
	\[\sum_{A\subseteq T\subseteq B}(-1)^{|T|-|A|}m(T)\geq 0.\]
	\item[(4)]  If $A\subseteq B \subseteq X$ and $rk^{*}(A)=rk^{*}(B)$, then
	\[\sum_{A\subseteq T\subseteq B}(-1)^{|T|-|A|}m(X\setminus T)\geq 0,\]
	where $rk^{*}(A)=|A|+rk(X\setminus A)-rk(X)$ and $rk^{*}(B)$ is similar. 
\end{description}

In \cite{Branden}, 
Br\"{a}nd\'{e}n and Moci generalized the
multivariate Tutte polynomial from matroids to arithmetic matroids. 
The \emph{multivariate arithmetic Tutte polynomial} of an arithmetic matroid  $\mathcal{M}=(X,rk,m)$ is defined by
\begin{align}\label{equ1.2}
	\mathcal{Z}_{\mathcal{M}}(q,\mathbf{v})=\sum_{A\subseteq X}m(A)q^{-rk(A)}\prod_{e\in A}v_e. 
\end{align}
In particular, if $v_e=v$ for each $e\in X$ in a multiplicity matroid $\mathcal{M}=(X,rk,m)$, then we write $\mathcal{Z}_{\mathcal{M}}(q,v)$ for $\mathcal{Z}_{\mathcal{M}}(q,\mathbf{v})$.

If we substitute $q$ and the variable $v_e$ of each $e\in X$ for $(x-1)(y-1)$ and $y-1$ respectively, and multiply by a prefactor $(x-1)^{r(\mathcal{M})}$ into Eq. (\ref{equ1.2}),
we obtain the \emph{arithmetic Tutte polynomial}
\[\mathfrak{M}_{\mathcal{M}}(x,y)=\sum_{A\subseteq X}m(A)(x-1)^{rk(X)-rk(A)}(y-1)^{|A|-rk(A)},\]
introduced by D'Adderio and Moci in \cite{D'AdderioMoci}. If $\mathcal{M}$ is only a multiplicity matroid, $\mathfrak{M}_{\mathcal{M}}(x,y)$ is called the \emph{multiplicity Tutte polynomial} of $\mathcal{M}$, introduced by  Moci in \cite{Moci}.

Kook, Reiner and Stanton \cite{Kook} and Etienne and Las Vergnas\cite{Etienne} found a well-known convolution formula for the Tutte polynomial $T_{M}(x,y)$ of a matroid $M$: 
\[T_{M}(x,y)=\sum_{A\subseteq X} T_{M/A}(x,0)T_{M|A}(0,y),\]
where $M/A$ and $M|A$ denote the contraction and restriction of $A$ from $M$, respectively. Kung \cite{Kung} generalized this formula to  subset-corank  polynomials which are related to
multivariate Tutte polynomial.

Motivated by the work of Kung in \cite{Kung}, we first obtain the convolution formulas for multivariate arithmetic Tutte polynomials of the product of two arithmetic matroids in this note. In particular,
the convolution formulas for the multivariate arithmetic Tutte polynomial
of an arithmetic matroid are given. Secondly, applying our results, several known convolution formulas are proved by a purely combinatorial
method. 
Finally, we obtain a convolution formula for the characteristic polynomial of an arithmetic matroid.

\section{Main results }
Let $\mathcal{M}=(X,rk,m)$ be a multiplicity matroid. For $T\subseteq X$, the restriction and contraction of $T$ from $\mathcal{M}$ were given in \cite{D'AdderioMoci}. 

The multiplicity matroid on $T$ with the rank function and the multiplicity obtained by restricting $rk$ and $m$ to subsets of $T$ respectively, denoted by $\mathcal{M}|T$, is called the \textit{restriction} of $\mathcal{M}$ to $T$.
The \emph{contraction} of $T$ from $\mathcal{M}$, denoted by $\mathcal{M}/T$, is the multiplicity matroid $(X\setminus T, rk_{\mathcal{M}/T}, m_{\mathcal{M}/T})$, where $rk_{\mathcal{M}/T}$ and $m_{\mathcal{M}/T}$ are defined by 
$$rk_{\mathcal{M}/T}(A)=rk(A\cup T)-rk(T)$$
and $$m_{\mathcal{M}/T}(A)=m(A\cup T)$$ 
for $A\subseteq X\setminus T$. 

Let $M=(X,rk)$ be a matroid. Recall that $\mathbf{v}_A=\{v_e:e\in A\}$. Let $u_e$ be another variable of $e\in A$ and $\mathbf{u}_A=\{u_e:e\in A\}$  for $A\subseteq X$. We define $(\mathbf{u}\mathbf{v})_{A}=\{u_ev_e:e\in A\}$ for $A\subseteq X$, and write $\mathbf{u}\mathbf{v}$ for 
$(\mathbf{u}\mathbf{v})_{X}$. 

For two multiplicity matroids $\mathcal{M}_1=(X,rk,m_1)$ and $\mathcal{M}_2=(X,rk,m_2)$ over a fixed underlying matroid $(X,rk)$, the \emph{product} of  $\mathcal{M}_1$ and  $\mathcal{M}_2$, denoted by $\mathcal{M}_1\bullet\mathcal{M}_2$, is defined by $(X,rk,m_1m_2)$, where $m_1m_2$ is the product of two multiplicity functions $m_1$ and $m_2$, given by $$m_1m_2(A)=m_1(A)m_2(A)$$ 
for $A\subseteq X$. 

Delucchi and Moci \cite{Delucchi} proved that
if both $\mathcal{M}_1$ and $\mathcal{M}_2$ are two arithmetic matroids, then $\mathcal{M}_1\bullet\mathcal{M}_2$ is also an arithmetic matroid. 

We first have the following convolution formula for the multivariate arithmetic Tutte polynomial of the product of two arithmetic matroids. 
\begin{theorem}\label{mtpc0}
	Let $\mathcal{M}_1=(X,rk,m_1)$ and $\mathcal{M}_2=(X,rk,m_2)$ be two arithmetic matroids, and let $\mathcal{M}=\mathcal{M}_1\bullet\mathcal{M}_2$. Then
	\begin{align*}
		\mathcal{Z}_{\mathcal{M}}(pq,\mathbf{uv})=&\sum_{T\subseteq X}p^{-rk(T)}\prod_{e\in T}(-u_e)\mathcal{Z}_{\mathcal{M}_1|T}(q,-\mathbf{v})\mathcal{Z}_{\mathcal{M}_2/T}(p,\mathbf{u})\\
		=&\sum_{T\subseteq X}p^{-rk(T)}\prod_{e\in T}(-u_e)\mathcal{Z}_{\mathcal{M}_2|T}(q,-\mathbf{v})\mathcal{Z}_{\mathcal{M}_1/T}(p,\mathbf{u}).
	\end{align*}
\end{theorem}
\begin{proof}
	For two subsets $A$ and $B$ of $X$, we have 
	\[\sum_{T:B\subseteq T\subseteq A}(-1)^{|T|-|B|}=
	\begin{cases}
		1, & \text{if } A=B,\\
		0, & \text{otherwise. }
	\end{cases}\]
	Therefore, we have
	\begin{align*}
		\mathcal{Z}_{\mathcal{M}}(pq,\mathbf{uv})&=\sum_{A,B:B\subseteq A\subseteq X}m_1(B)m_2(A)p^{-rk(A)}q^{-rk(B)}\left(\prod_{e\in A}u_e\right)\left(\prod_{e\in B}v_e\right)\sum_{T: B\subseteq T\subseteq A}(-1)^{|T|-|B|}\\
		&=\sum_{T: T\subseteq X}(-1)^{|T|}p^{-rk(T)}\left(\prod_{e\in T}u_e\right)\left(\sum_{B: B\subseteq T}(-1)^{|B|}m_{1}(B)q^{-rk(B)}\prod_{e\in B}v_e\right)\cdot\\
		&\ \ \ \left(\sum_{A: T\subseteq A\subseteq X}m_2(A)p^{-rk_{M/T}(A\setminus T)}\prod_{e\in A\setminus T}u_e\right)\\
		&=\sum_{T: T\subseteq X}p^{-rk(T)}\left(\prod_{e\in T}(-u_e)\right)\left(\sum_{B: B\subseteq T}m_{1}(B)q^{-rk(B)}\prod_{e\in B}(-v_e)\right)\cdot\\
		&\ \ \ \left(\sum_{A: T\subseteq A\subseteq X}m_2(A)p^{-rk_{M/T}(A\setminus T)}\prod_{e\in A\setminus T}u_e\right)\\
		&=\sum_{T: T\subseteq X}p^{-rk(T)}\left(\prod_{e\in T}(-u_e)\right)\mathcal{Z}_{\mathcal{M}_1|T}(q,-\mathbf{v})\mathcal{Z}_{\mathcal{M}_2/T}(p,\mathbf{u}).
	\end{align*}
	
	The second equation holds since $\mathcal{M}=\mathcal{M}_1\bullet\mathcal{M}_2=\mathcal{M}_2\bullet\mathcal{M}_1$.
\end{proof}

We write $\mathcal{Z}_{\mathcal{M}}(q,uv)$ for $\mathcal{Z}_{\mathcal{M}}(q,\mathbf{uv})$ if $u_e=u$ and $v_e=v$ for each $e\in X$ in an arithmetic matroid $\mathcal{M}=(X,rk,m)$ .
We have the following specialization of Theorem \ref{mtpc0}. 
\begin{theorem}\label{mtpc1}
	Let $\mathcal{M}_1=(X,rk,m_1)$ and $\mathcal{M}_2=(X,rk,m_2)$ be two arithmetic matroids, and let $\mathcal{M}=\mathcal{M}_1\bullet\mathcal{M}_2$. Then
	\begin{align*}
		\mathcal{Z}_{\mathcal{M}}(pq,uv)=&\sum_{T\subseteq X}p^{-rk(T)}(-u)^{|T|}\mathcal{Z}_{\mathcal{M}_1|T}(q,-v)\mathcal{Z}_{\mathcal{M}_2/T}(p,u)\\
		=&\sum_{T\subseteq X}p^{-rk(T)}(-u)^{|T|}\mathcal{Z}_{\mathcal{M}_2|T}(q,-v)\mathcal{Z}_{\mathcal{M}_1/T}(p,u).
	\end{align*}
\end{theorem}

For an arithmetic matroid $\mathcal{M}=(X,rk,m)$, we use $M$ to denote the underlying matroid $(X,rk)$ conveniently. Taking one of two multiplicity functions $m_1$ and $m_2$ in Theorems \ref{mtpc0} and \ref{mtpc1} to be trivial,  we obtain:

\begin{corollary}\label{mtpc2}
	Let $\mathcal{M}=(X,rk,m)$ be an arithmetic matroid. Then
	\begin{align*}
		\mathcal{Z}_{\mathcal{M}}(pq,\mathbf{uv})=&\sum_{T\subseteq X}p^{-rk(T)}\prod_{e\in T}(-u_e)\mathcal{Z}_{\mathcal{M}|T}(q,-\mathbf{v})Z_{M/T}(p,\mathbf{u})\\
		=&\sum_{T\subseteq X}p^{-rk(T)}\prod_{e\in T}(-u_e)Z_{M|T}(q,-\mathbf{v})\mathcal{Z}_{\mathcal{M}/T}(p,\mathbf{u}).
	\end{align*}
\end{corollary}

\begin{corollary}\label{CoroMaT}
	Let $\mathcal{M}=(X,rk,m)$ be an arithmetic matroid. Then
	\begin{align*}
		\mathcal{Z}_{\mathcal{M}}(pq,uv)=&\sum_{ T\subseteq X}p^{-rk(T)}(-u)^{|T|}\mathcal{Z}_{\mathcal{M}|T}(q,-v)Z_{M/T}(p,u)\\
		=&\sum_{T\subseteq X}p^{-rk(T)}(-u)^{|T|}Z_{\mathcal{M}|T}(q,-v)\mathcal{Z}_{M/T}(p,u).
	\end{align*}
\end{corollary}

A generalization of the well-known convolution formula for the Tutte polynomial of a matroid was shown by Kung in \cite{Kung}; for details, see also \cite[Theorem 12.25]{Ellis-Monaghan} or \cite[Theorem 5.3]{Wang0}. Recently, an analogous
formula for the arithmetic Tutte polynomial of the product of two arithmetic matroids was
obtained
by Dupont, Fink and Moci in \cite{Dupont}. As an application of Theorem \ref{mtpc1}, we now give it a purely combinatorial
proof.

\begin{corollary}\emph{\cite{Dupont}}
	Let $\mathcal{M}_1=(X,rk,m_1)$ and $\mathcal{M}_2=(X,rk,m_2)$ be two arithmetic matroids, and let $\mathcal{M}=\mathcal{M}_1\bullet\mathcal{M}_2$. Then
	\begin{align*}
		&\mathfrak{M}_{\mathcal{M}}(1+ab,1+cd)\\
		&=\sum_{A\subseteq X}a^{rk(X)-rk(A)}(-d)^{|A|-rk(A)} \mathfrak{M}_{\mathcal{M}_1|A}(1-a,1-c)\mathfrak{M}_{\mathcal{M}_2/A}(1+b,1+d)\\
		&=\sum_{A\subseteq X}a^{rk(X)-rk(A)}(-d)^{|A|-rk(A)} \mathfrak{M}_{\mathcal{M}_2|A}(1-a,1-c)\mathfrak{M}_{\mathcal{M}_1/A}(1+b,1+d).
	\end{align*} 
\end{corollary}

\begin{proof}
	Recall that       \begin{align}\label{equ}
		\mathfrak{M}_{\mathcal{M}}(x,y)=(x-1)^{rk(X)}\mathcal{Z}_{\mathcal{M}}((x-1)(y-1),y-1).
	\end{align}
	Then we have
	\[\mathfrak{M}_{\mathcal{M}}(1+ab,1+cd)=(ab)^{rk(X)}\mathcal{Z}_{\mathcal{M}}(abcd,cd).\]
	Taking $p=bd$, $q=ac$, $u=d$ and $v=c$, 
	by Theorem \ref{mtpc1}, we have
	\begin{align*}
		\mathcal{Z}_{\mathcal{M}}(abcd,cd)=&\sum_{A\subseteq X}(bd)^{-rk(A)}(-d)^{|A|}\mathcal{Z}_{\mathcal{M}_1|A}(ac,-c)\mathcal{Z}_{\mathcal{M}_2/A}(bd,d).
	\end{align*}
	By Eq. (\ref{equ}), taking $x=1-a$ and $y=1-c$, we have
	\[\mathcal{Z}_{\mathcal{M}_1|A}(ac,-c)=(-a)^{-rk(A)}\mathfrak{M}_{\mathcal{M}_1|A}(1-a,1-c),\]
	and taking $x=1+b$ and $y=1+d$, we have
	\[\mathcal{Z}_{\mathcal{M}_2/A}(bd,d)=b^{-rk(\mathcal{M}_2/A)}\mathfrak{M}_{\mathcal{M}_2/A}(1+b,1+d).\]
	Therefore
	\begin{align*}
		&\mathfrak{M}_{\mathcal{M}}(1+ab,1+cd)\\
		&=(ab)^{rk(X)}\sum_{A\subseteq X}(bd)^{-rk(A)}(-d)^{|A|}(-a)^{-rk(A)} b^{-rk(\mathcal{M}_2/A)}\mathfrak{M}_{\mathcal{M}_1|A}(1-a,1-c)\mathfrak{M}_{\mathcal{M}_2/A}(1+b,1+d)\\
		&=b^{rk(X)}\sum_{A\subseteq X}b^{-rk(A)-rk(\mathcal{M}_2/A)}a^{rk(X)-rk(A)}(-d)^{|A|-rk(A)} \mathfrak{M}_{\mathcal{M}_1|A}(1-a,1-c)\mathfrak{M}_{\mathcal{M}_2/A}(1+b,1+d).
	\end{align*}
	Note that $rk(X)=rk(A)+rk(\mathcal{M}_2/A)$. Thus the first equation
	holds.
	The second equation holds since $\mathcal{M}=\mathcal{M}_1\bullet\mathcal{M}_2=\mathcal{M}_2\bullet\mathcal{M}_1$. 
\end{proof}

We note that four axioms in the definition of arithmetic matroids are not used in the proofs. Therefore the above results also hold for multiplicity matroids. 

Recently, Backman and Lenz \cite{Backman} and Dupont, Fink and Moci \cite{Dupont}
obtained the following formula. We apply Theorem \ref{mtpc1} to give it a new and simple proof.

\begin{corollary}\emph{\cite{Backman,Dupont}}\label{BackmanCF}
	Let $\mathcal{M}_1=(X,rk,m_1)$ and $\mathcal{M}_2=(X,rk,m_2)$ be two multiplicity matroids, and let $\mathcal{M}=\mathcal{M}_1\bullet\mathcal{M}_2$. Then
	\begin{align*}
		\mathfrak{M}_{\mathcal{M}}(x,y)&=\sum_{A\subseteq X} \mathfrak{M}_{\mathcal{M}_1|A}(0,y)\mathfrak{M}_{\mathcal{M}_2/A}(x,0)\\
		&=\sum_{A\subseteq X} \mathfrak{M}_{\mathcal{M}_2|A}(0,y)\mathfrak{M}_{\mathcal{M}_1/A}(x,0).
	\end{align*}
\end{corollary}
\begin{proof}
	By setting $p=1-x$, $q=1-y$, $u=-1$ and $v=1-y$ and multiplying by $(x-1)^{rk(X)}$ into two equations in Theorem \ref{mtpc1}, we have 
	\begin{align*}
		&(x-1)^{rk(X)}\mathcal{Z}_{\mathcal{M}}((1-x)(1-y),y-1)\\
		&=\sum_{ A\subseteq X}(-1)^{-rk(A)}\mathcal{Z}_{\mathcal{M}_1|A}(1-y,y-1)\cdot(x-1)^{rk(X)-rk(A)}\mathcal{Z}_{\mathcal{M}_2/A}(1-x,-1)\\
		&=\sum_{A\subseteq X}(-1)^{-rk(A)}\mathcal{Z}_{\mathcal{M}_2|A}(1-y,y-1)\cdot(x-1)^{rk(X)-rk(A)}\mathcal{Z}_{\mathcal{M}_1/A}(1-x,-1).
	\end{align*} 
	Note that $rk(X)-rk(A)=rk(M/A)$. Thus, by Eq. (\ref{equ}), the result is established.   
\end{proof}

In \cite{Backman}, Backman and Lenz also obtained the following  formula for a multiplicity matroid $\mathcal{M}=(X,rk,m)$:
\begin{align*}
	\mathfrak{M}_{\mathcal{M}}(x,y)&=\sum_{A\subseteq X} \mathfrak{M}_{\mathcal{M}|A}(0,y)T_{M/A}(x,0)\\
	&=\sum_{A\subseteq X} T_{M|A}(0,y)\mathfrak{M}_{\mathcal{M}/A}(x,0).
\end{align*}
Using the technique similar to the proof of Corollary \ref{BackmanCF}, Corollary \ref{CoroMaT} gives a purely combinatorial
proof of this formula. Certainly, it can be also proved by taking one of two multiplicity functions $m_1$ and $m_2$ in Corollary \ref{BackmanCF} to be trivial.

In \cite{Wang}, Wang, Yeh and Zhou defined the characteristic polynomial $\chi_{\mathcal{M}}(\lambda)$ of an arithmetic matroid $\mathcal{M}=(X,rk,m)$ as follows:
\[\chi_{\mathcal{M}}(\lambda)=\sum_{A\subseteq X}(-1)^{|A|}m(A)\lambda^{rk(X)-rk(A)}.\] 
Note that if the multiplicity $m$ of an arithmetic matroid $\mathcal{M}=(X,rk,m)$ is trivial, then the characteristic polynomial $\chi_{\mathcal{M}}(\lambda)$ of $\mathcal{M}$ specializes to the classical characteristic polynomial $\chi_{M}(\lambda)$ of $M$. It is easy to see that 
\begin{align}\label{chrelation1}
	\chi_{\mathcal{M}}(\lambda)=\lambda^{rk(X)}\mathcal{Z}_{\mathcal{M}}(\lambda,-1)
\end{align}     
and 
\begin{align}\label{chrelation2}
	\chi_{M}(\lambda)=\lambda^{rk(X)}Z_{M}(\lambda,-1).
\end{align}     

Applying Corollary \ref{CoroMaT},
we have the following convolution formula for the characteristic polynomial of an arithmetic matroid.\begin{theorem}
	Let $\mathcal{M}=(X,rk,m)$ be an arithmetic matroid. Then
	\begin{align*}
		\chi_{\mathcal{M}}(\lambda\xi)&=\sum_{A\subseteq X}\lambda^{rk(X)-rk(A)}\chi_{\mathcal{M}|A}(\lambda)\chi_{M/A}(\xi)\\
		&=\sum_{A\subseteq X}\lambda^{rk(X)-rk(A)}\chi_{M|A}(\lambda)\chi_{\mathcal{M}/A}(\xi).
	\end{align*}
\end{theorem}
\begin{proof}
	By setting $p=\xi$, $q=\lambda$, $u=-1$ and $v=1$ and multiplying by $(\lambda\xi)^{rk(X)}$ into two equations in Corollary \ref{CoroMaT}, we have 
	\begin{align*}
		(\lambda\xi)^{rk(X)}\mathcal{Z}_{\mathcal{M}}(\lambda\xi,-1)=&^{}\sum_{ A\subseteq X}\lambda^{rk(X)}\mathcal{Z}_{\mathcal{M}|A}(\lambda,-1)\cdot\xi^{rk(X)-rk(A)}Z_{M/A}(\xi,-1)\\
		=&\sum_{A\subseteq X} \lambda^{rk(X)}Z_{\mathcal{M}|A}(\lambda,-1)\cdot\xi^{rk(X)-rk(A)}\mathcal{Z}_{M/A}(\xi,-1).
	\end{align*} 
	Note that $rk(X)=rk(A)+rk(M/A)$. Thus the equations hold from Eqs. (\ref{chrelation1}) and (\ref{chrelation2}).  
\end{proof}

\section*{Acknowledgements}
This work is supported by National Natural Science Foundation of China
(No. 12171402).


\end{document}